\documentclass{article}
\usepackage{amsmath}
\usepackage{amsthm,amssymb,enumerate,hyperref}
\usepackage[text={6.5in,8in},centering]{geometry}
\usepackage{url}
\usepackage{xcolor}
\hypersetup{colorlinks,linkcolor={red!50!black},citecolor={red!50!black},urlcolor={blue!60!black}}
\theoremstyle{plain}
\newtheorem{theorem}{Theorem}
\newtheorem{proposition}[theorem]{Proposition}

\newtheorem{lemma}[theorem]{Lemma}
\theoremstyle{definition}
\newtheorem{remark}[theorem]{Remark}
\newtheorem*{remark*}{Remark}
\newtheorem*{finalremarks}{Final remarks}

\newtheorem{corollary}[theorem]{Corollary}
\newcommand{\N}{\mathbb{N}}
\newcommand{\R}{\mathbb{R}}

\begin{document}

\title{A rate of metastability for the Halpern type Proximal Point Algorithm}
\author{Pedro Pinto${}^{a}$\\[2mm]
	\footnotesize ${}^{a}$ Department of Mathematics, Technische Universit{\"a}t Darmstadt,\\ 
	\footnotesize Schlossgartenstrasse 7, 64289 Darmstadt, Germany \\ [2mm]
	\footnotesize E-mail:  \protect\url{pinto@mathematik.tu-darmstadt.de}
}
\date{}
\maketitle

\begin{abstract}
Using proof-theoretical techniques, we analyze a proof by H.-K. Xu regarding a result of strong convergence for the Halpern type proximal point algorithm. We obtain a rate of metastability (in the sense of T. Tao) and also a rate of asymptotic regularity for the iteration. Furthermore, our final quantitative result bypasses the need of the sequential weak compactness argument present in the original proof. This elimination is reflected in the extraction of primitive recursive quantitative information. This work follows from recent results in Proof Mining regarding the removal of sequential weak compactness arguments.\\

\noindent {\em Keywords:} Proximal point algorithm; Sequential weak compactness; Proof mining; Metastability; Rates of asymptotic regularity.\\

\noindent  {\it Mathematics Subject Classification 2010}: 47H05; 47H09; 47J25; 03F10.

\end{abstract}

\section{Introduction}
Let $X$ be a real Hilbert space and $\mathsf{A}:X\to 2^{X}$ be a maximal monotone operator on $X$. We denote by $S$ the zero set of $\mathsf{A}$. One of the major problems in the theory of maximal operators is how to find a point  in $S$. The relevance of this search for zeros stems from the fact that many problems in nonlinear analysis and optimization theory can be formulated as a question of finding a zero for specific maximal monotone operators. An important tool in finding such zeros of maximal monotone operators is Rockafellar's proximal point algorithm: given an initial guess $x_0$, a sequence of positive real numbers $(\beta_n)$ and accepting a sequence of possible errors $(e_n)$, the algorithm is defined recursively by $x_{n+1}:=J_{\beta_n}(x_n)+e_n$.
Rockafellar showed in \cite{rockafellar1976monotone} that, provided $(\beta_n)$ is bounded away from zero and $(\|e_n\|)$ is a summable sequence, the proximal point algorithm must converge weakly to a zero point. In \cite{guler1991convergence}, G{\"u}ler argued that the algorithm may fail to be strongly convergent already in infinite-dimensional Hilbert spaces. This gave rise to several modifications to the proximal point algorithm in an attempt to try ensure a strong convergence result. Motivated by the success of Halpern iterations for nonexpansive mappings in fixed point theory (see e.g. \cite{halpern1967fixed,wittmann1992approximation,bauschke1996approximation}), Kamimura and Takahashi in \cite{kamimura2000approximating} and independently Xu in \cite{xu2002iterative} introduced the dubbed Halpern type proximal point algorithm: Consider $(\alpha_n)\subset\, ]0,1[\,$, $(\beta_n)\subset \R^+$ sequences of real numbers, $x_0\in X$ an initial guess and $(e_n)\subset X$ an error sequence. Then, the Halpern type proximal point algorithm is recursively defined by the convex linear combination
\begin{equation}\label{hppa}\tag{HPPA}
x_{n+1}:=\alpha_nx_0+(1-\alpha_n)(J_{\beta_n}(x_n) + e_n).
\end{equation}

In \cite{xu2002iterative} Xu showed that, under some conditions on the parameters, the algorithm \eqref{hppa} strongly converges to a zero point, as follows.
\begin{theorem}[Xu {\cite[Theorem 5.1]{xu2002iterative}}]\label{original_xu}
	Let $(x_n)$ be generated by \eqref{hppa} and assume that the following conditions hold
	\begin{enumerate}[\rm(C$1$)]
		\item\label{c1} $\lim \alpha_n = 0$
		\item\label{c2} $\sum \alpha_n =\infty$
		\item\label{c3} $\lim \beta_n =\infty$
		\item\label{c4} $\sum \|e_n\|< \infty$.
	\end{enumerate}
	Then $(x_n)$ strongly converges to a zero of $\textsf{A}$, the closest one to $x_0$.
\end{theorem}

The main goal of this paper is to analyze Xu's original proof and obtain a quantitative version of Theorem~\ref{original_xu}. Techniques from proof theory have been successfully employed in the analysis of mathematical proofs as a way to obtain stronger results: either by relaxing the assumptions necessary for the proof or by exhibiting new computational information previously hidden in the main arguments of the original proof. This research program has been called Proof Mining and was greatly developed by Ulrich Kohlenbach and his collaborators for the last twenty five years with applications to results in many areas of mathematics (for a comprehensive reference see \cite{kohlenbach2008applied} and for a more recent overview see \cite{kohlenbach2017recent}\cite{kohlenbach2018kreisel}). In this quantitative analysis, we will extract a bound on the metastability of the sequence $(x_n)$ in the sense of Terence Tao \cite{tao2008structure}\cite{tao2008norm}, i.e. extract a functional $\Phi: \N \times \N^{\N}\to \N$ such that
\begin{equation}\label{goal}
\forall k\in \N \forall f\in \N^{\N} \exists n\leq \Phi(k,f)\, \forall i, j\in [n, f(n)]\, \left( \|x_i-x_j\|\leq \frac{1}{k+1}\right).
\end{equation}

Notice that the Cauchy property for the sequence $(x_n)$ can be obtained (ineffectively) from \eqref{goal}. While in general one cannot obtain information on the rate of the Cauchy property, the extraction of the functional $\Phi$ is guaranteed by the theoretical aspects of the proof interpretation underlying the analysis -- and is achieved in Theorems~\ref{main1} and \ref{main2} bellow. Here we will be guided by the bounded functional interpretation \cite{ferreira2005bounded,ferreira2009injecting}. This quantitative analysis was carried out in the context of the author's PhD thesis~\cite{pinto2019proof} and continues the application of this functional interpretation to concrete cases of Proof Mining \cite{ferreira2019removal,dinis2019metastability}. Nevertheless, as is usual with proof mining results, the proof-theoretical techniques are only employed as an intermediate step and no particular logical knowledge is required to understand the main results in this paper. This analysis also follows several quantitative studies into the proximal point algorithm and variants thereof (\cite{leucstean2018abstract,leustean2018effective,leucstean2018application,kohlenbach2018quantitative,kohlenbach2019quantitative,dinis2019metastability}).\\
In the next section, we briefly sketch Xu's original proof and identify the most delicate steps of the analysis. We look at the analysis of the main arguments used in the proof and give a quantitative form to a lemma by Xu \cite[lemma 2.5]{xu2002iterative} which is particularly useful for our analysis. In the following section, we proceed with the extraction of a rate of asymptotic regularity and a bound on the metastability of \eqref{hppa}.

\section{Preliminaries}

Let $X$ be a real Hilbert space and consider a multi-valued operator $\textsf{A}:X\to 2^X$ to be maximal monotone, i.e. $\textsf{A}$ satisfies the monotonicity property $ \langle x-x', y-y'\rangle \geq 0$, for all $x$, $x'\in X$, $y\in \textsf{A}(x)$, $y'\in \textsf{A}(x')$,
and its graph cannot be extended while keeping this property. We denote by $S$ the set of zeros of $\textsf{A}$, $$S:=\{x\in X\, :\, 0\in \textsf{A}(x)\}.$$ It is well-known that the set $S$ is closed and convex, and we will henceforth assume it to be nonempty. For any positive real number $\beta$, the resolvent function $J_{\beta}$ defined by $J_{\beta}:=(Id+\beta \textsf{A})^{-1}$ is a single-valued nonexpansive mapping on $X$, i.e. ${\|J_{\beta}(x)-J_{\beta}(y)\|\leq \|x-y\|}$, for all ${x,y\in X}$. Furthermore, it is easily seen that the set of fixed points of a resolvent function coincides with the zero set $S$. For a comprehensive reference on maximal monotone operators in Hilbert spaces, we refer to \cite{bauschke2017convex}.\\

The structure of Xu's proof of theorem \ref{original_xu} is as follows:
\begin{enumerate}[ Step 1.]
	\item \underline{$(x_n)$ is bounded:} This initial fact is easily seen using an inductive reasoning.
	\item \underline{Projection onto $S$:} Xu considers $P_S(x_0)$, the metric projection point of $x_0$ onto $S$, which is well defined since $S$ is a nonempty, closed and convex set. The usefulness of the point $P_S(x_0)$ in the proof lies in the following crucial variational inequality: $\forall y\in S\, \left( \langle x_0-P_S(x_0), y-P_S(x_0)\rangle \leq 0\right)$.
	\item \underline{$\limsup \,\langle x_0-P_S(x_0), x_n-P_S(x_0)\rangle \leq 0$:} This is argued using a weak sequential compactness argument in conjunction with the demiclosedness principle.
	\item \underline{$x_n\to P_S(x_0)$:} The last step of the proof centers on the application of the following lemma:
	\begin{lemma}[Xu {\cite[Lemma 2.5]{xu2002iterative}}]\label{xu_seq_reals_1}
		Let $(s_n)$ be a sequence of nonnegative real numbers and assume that for any $n\in \N$
		\begin{equation}\label{xu_lem_main_ass}
		s_{n+1}\leq (1-\alpha_n)s_n+\alpha_nr_n + \gamma_n,
		\end{equation}
		\noindent with $(\alpha_n)\subset\, ]0,1[$, $(r_n)$ and $(\gamma_n)\subset \R^+_0$ are sequences of real numbers such that: \rm(C{\ref{c2}}) holds, $\limsup r_n\leq 0$ and $\sum \gamma_n <\infty$. Then $\lim s_n =0$.
		\end{lemma}
\end{enumerate}

	Step 1. and Step 4. pose no problem for the analysis. Regarding Step 2. and Step 3., we will see that it is possible to replace the projection argument with a weaker statement, to completely bypass the need of weak compactness and in the end still obtain a metastability result for the sequence $(x_n)$. These changes to the principles needed in the proof have a clear impact in simplifying the extracted information (for logicians: namely by ensuring that we only need primitive recursive functionals in the sense of G\"odel and avoiding the need of functionals defined by bar recursion).\\

It will be useful to recall the notion of monotone functional for two particular cases. We rely on the strong majorizability relation introduced by Bezem in \cite{bezem1985strongly}. First, given functions $f,g:\N\to\N$, we say that $f$ majorizes $g$, writing $g\leq^* f$, when
\begin{equation*}
\forall n\in \N\, \forall m\leq n\, \left( g(m)\leq f(n) \land f(m)\leq f(n)\right).
\end{equation*}
Then, $f$ is said to be monotone if $f\leq^*f$ which can easily be seen to correspond to the usual notion of being increasing: $\forall n\in\N\, \left( f(n)\leq f(n+1)\right)$. Also notice that for any function $f$, we have $f\leq^* f^{\rm maj}$, where $f^{\rm maj}$ is the monotone function defined by $f^{\rm maj}(n):=\max\{f(m) : m\leq n\}$. A functional $\psi:\N\times \N^{\N}\to \N$ is said to be monotone if for any $n, m\in\N$ and $f, g:\N \to \N$,
\begin{equation*}
\left(m\leq n \land g\leq^* f\right) \to \psi(m,g)\leq \psi(n,f).
\end{equation*}

In order to give a quantitative version of Theorem~\ref{original_xu}, we must give quantitative meaning to the conditions \rm(C{\ref{c1}})-\rm(C{\ref{c4}}):

\begin{enumerate}[{\rm (Q$1$)}]
	\item\label{q1} ${\rm a}:\N\to \N$ is a rate of convergence towards zero for $(\alpha_n)$, i.e. $\displaystyle{\forall k\in\N\,\forall n\geq {\rm a}(k) \, \left(\alpha_n\leq \frac{1}{k+1}\right)}$;
	\item\label{q2} ${\rm A}:\N\to \N$ is a rate of divergence for $(\sum \alpha_n)$, i.e. $\displaystyle{\forall k\in\N \,\left(\sum_{i=0}^{{\rm A}(k)}\alpha_i\geq k\right)}$;
	\item\label{q3} ${\rm B}:\N\to \N$ is a rate of divergence for $(\beta_n)$, i.e. $\displaystyle{\forall k\in \N\, \forall n\geq {\rm B}(k)\, \left( \beta_n \geq k\right)}$;
	\item\label{q4} ${\rm E}:\N\to\N$ is a \emph{Cauchy rate} for $\left(\sum_{i=0}^{n} \|e_i\|\right)$, i.e., $\displaystyle{\forall k\in \N\,\forall n\in \N\, \left( \sum_{i={\rm E}(k)+1}^{{\rm E}(k)+n} \|e_i\|\leq \frac{1}{k+1}\right)}$;
\end{enumerate}

Without loss of generality, we can assume all the functions above to be monotone. In fact, the conditions will still hold true if the functions were replaced by their $(\cdot)^{\rm maj}$-counterpart.

\subsection{Projection and weak compactness}

Quantitative analyses of the projection argument have been carried out before, e.g. in \cite{kohlenbach2011quantitative}\cite{kohlenbach2010logical} (guided by the monotone functional interpretation) and more recently, in a different way in \cite{ferreira2019removal}, using the bounded functional interpretation. Consider $C$ a closed and convex subset of $X$, $x_0$ a point in $C$, $T:C\to C$ a nonexpansive mapping and $F$ the fixed-point set of $T$, i.e. $F:=\{x\in C\, :\, T(x)=x\}$. In \cite{kohlenbach2011quantitative}, Kohlenbach remarked that instead of the original projection statement
\begin{equation}\label{proj0}
\exists x\in F\, \forall k\in \N \, \forall y\in F \, \left( \|x-x_0\|\leq \|y-x_0\|+\frac{1}{k+1} \right),
\end{equation}

the following weaker version is already sufficient for quantitative analyses
\begin{equation}\label{proj}
\forall k\in \N \, \exists x\in F\, \forall y\in F \, \left( \|x-x_0\|\leq \|y-x_0\|+\frac{1}{k+1} \right).
\end{equation}
\indent The relevance of this point is in the fact the weaker version can be shown using simply an inductive argument (while the original statement requires a strong form of choice), which translates to simpler quantitative information.

While the analysis of \eqref{proj} done in \cite{ferreira2019removal} consider the set $C$ to be bounded, here we don't have that condition and instead know $F$ to be nonempty. Nevertheless, with this hypothesis, we can recover a `boundedness' condition by restricting the projection onto fixed points inside a closed ball with a certain radius.\\
\indent Let $p$ be a zero of $\textsf{A}$. For any $N\in \N$, let $B_{N}:=\left\{x\in X\, |\, \|x-p\|\leq N\,\right\}$ denote the closed ball centered at $p$ with radius $N$. Looking again at \eqref{proj0}, it is clear that if $N\geq 2\|x_0-p\|$, then one can equivalently replace $F$ with $F\cap B_N$. Now, weakening this restriction, we arrive at the following statement
\begin{equation}\label{proj1}
\forall k\in \N \, \exists x\in F\cap B_{N}\, \forall y\in F\cap B_{N} \, \left( \|x-x_0\|\leq \|y-x_0\|+\frac{1}{k+1} \right).
\end{equation}

We have the following quantitative version of \eqref{proj1}.
\begin{proposition}\label{q_proj1}
Let $N\in \N\setminus \{0\}$ be such that $N\geq 2\|x_0-p\|$ for some $p\in S$.\\
For any natural number $k$ and monotone function $f:\N \to \N $, there are $n \leq f^{(r)}(0)$ and $x\in C\cap B_{N}$ such that
$$\|T(x)-x\| \leq \frac{1}{f(n)+1}$$
and
$$\forall y \in C\cap B_{N}\, 
\left(\|T(y)-y \|\leq \frac{1}{n+1} \to \|x-x_0\|^2 \leq \|y-x_0\|^2 + \frac{1}{k+1}\right),$$
where $r:=r(N,k):=N^2(k+1)$ and $f^{(r)}$ is the $r$-th fold composition of $f$.
\end{proposition}

\begin{proof}
	It is essentially the proof of \cite[proposition 3.1]{ferreira2019removal} with two easy observations: first, see that we can take the point $p\in S$ for the initial element of the sequence defined there; second, in place of a bound on the diameter of the set $C$ -- in our case $C\cap B_N$ -- one can instead work with a bound on the norm $\|x_0-p\|$.
\end{proof}

\begin{remark}\hfill
	\begin{enumerate} 
	\item It is possible to work without the point $p$ and instead only with the assumption that on some ball there are arbitrarily good almost fixed points -- which is an easy fact (see Lemma~\ref{seq_bounded}, Proposition~\ref{asymp_reg} and also Lemma~\ref{same_fix_set_ab}).
	\item One may question the need to restrict the projection argument to a bounded set. However, this innocuous additional step justifies the logical form in the previous proposition (as was explained in \cite{ferreira2019removal}). The simplicity of this result contrast with the stronger one in \cite[lemma 2.4]{kohlenbach2011quantitative} and yet it suffices to our quantitative analysis.
	\end{enumerate}
\end{remark}

The following two results are essentially due to Kohlenbach \cite{kohlenbach2011quantitative}. For any $u,v\in X$ and  $t\in [0,1]$, consider the convex linear combination $q_t(u,v):=(1-t)u+tv$. 

\begin{lemma}\label{convex_F}
	For all $N\in\N\setminus \{0\}$, $k\in \N$ and $x_1,x_2\in C\cap B_N$,
	\begin{equation*}
	\bigwedge_{j=1}^2 \|T(x_j)-x_j\|\leq \frac{1}{24N(k+1)^2}\to \forall t\in [0,1]\, \left(\|T(q_t(x_1,x_2))-q_t(x_1,x_2)\|\leq \frac{1}{k+1}\right).
	\end{equation*}
\end{lemma}

\begin{lemma}\label{relation_to_innerp}
	For all $N\in\N\setminus \{0\}$, $k\in \N$ and $x, y\in C\cap B_N$,
	\begin{equation*}
	\forall t\in [0,1]\, \left( \|x-x_0\|^2\leq \|q_t(x,y)-x_0\|^2+\frac{1}{4N^2(k+1)^2}\right) \to \langle x_0-x, y-x\rangle \leq \frac{1}{k+1}.
	\end{equation*}
\end{lemma}

Lemma~\ref{convex_F} corresponds to \cite[lemma 2.3]{kohlenbach2011quantitative} with $d$, $\frac{\varepsilon^2}{16d}$ and $\varepsilon$ replaced by $2N$,  $\frac{1}{12(2N)(k+1)^2}$ and $\frac{1}{k+1}$, respectively. For Lemma~\ref{relation_to_innerp} we are using \cite[lemma 2.7]{kohlenbach2011quantitative} with $d=2N\geq \|x-y\|$ and with $\frac{1}{(2N)^2(k+1)^2}$ and $\frac{1}{k+1}$ replacing $\frac{\varepsilon^2}{2d^2}$ and $\varepsilon$, respectively. As in \cite[proposition 3.2]{ferreira2019removal}, now using Lemma~\ref{convex_F} and Lemma~\ref{relation_to_innerp}, we derive the following quantitative result.

\begin{proposition}\label{q_proj2}
	Let $N\in \N\setminus \{0\}$ be such that $N\geq 2\|x_0-p\|$ for some $p\in S$.\\
	For any natural number $k$ and monotone function $f:\N \to \N $, there are $n \leq 24N(w_{f,N}^{(R)}(0)+1)^2$ and $x\in C\cap B_N$ such that
	$$\|T(x)-x \| \leq \frac{1}{f(n)+1} \, \land \, \forall y\in C\cap B_N 
	\left(\|T(y)-y\|\leq \frac{1}{N+1} \to \langle x_0-x,y-x\rangle \leq \frac{1}{k+1}\right),$$ 
	with $R:=R(N,k):=4N^4(k+1)^2$ and $w_{f,N}(m):=\max\{ f(24N(m+1)^2),\, 24N(m+1)^2 \}$.
\end{proposition}

	In the remainder of this quantitative analysis, the set $C$ will be the entire Hilbert space $X$ and we will fix a particular resolvent function for the nonexpansive map $T$ (any will do since the corresponding set of fixed points is always the set $S$). We choose to work with the nonexpansive function $J:=J_1:=(Id+\textsf{A})^{-1}$ -- it is possible to work with $J_{\gamma}$ for an arbitrary real number $\gamma >0$ if we additionally consider a natural number $n$ that is an upper bound on the value of $\gamma$. The relevant result takes then the following form which contains the crucial content of the projection argument used in the proof of Theorem~\ref{original_xu}.

\begin{corollary}\label{proj_J}
	Let $\textsf{A}:X\to 2^X$ a maximal monotone operator on a Hilbert space $X$. Assume that the set $S$ of zeros of $\textsf{A}$ is nonempty and consider the resolvent function $J:=(Id+\textsf{A})^{-1}$. Take $x_0\in X$ and $N\in \N\setminus \{0\}$ a natural number satisfying $N\geq 2\|x_0-p\|$ for some point $p\in S$.
	For any $k\in \N$ and monotone function $f:\N \to \N $, there are $n \leq 24N(w_{f,N}^{(R)}(0)+1)^2$ and $x\in B_N$ such that
	$$\|J(x)-x \| \leq \frac{1}{f(n)+1} \, \land \, \forall y\in B_N 
	\left(\|J(y)-y\|\leq \frac{1}{n+1} \to \langle x_0-x,y-x\rangle \leq \frac{1}{k+1}\right),$$ 
	with $R:=R(N,k):=4N^4(k+1)^2$ and $w_{f,N}:=\max\{ f(24N(m+1)^2),\, 24N(m+1)^2 \}$.
\end{corollary}
	
	In the original proof (step 3), the sequential weak compactness argument is used to show
	\begin{equation}\label{swc0}
	\forall k\in \N \,\exists n\in\N \,\forall i\geq n\, \left(\langle x_0-P_S(x_0), x_i-P_S(x_0)\rangle\leq \frac{1}{k+1}\right),
	\end{equation}
	where $P_S(x_0)$ is the projection point of $x_0$ onto $S$. In fact, a weaker statement already suffices for the quantitative analysis, which is in line with the fact that we don't have the point $P_S(x_0)$ but only approximations to it in the sense of \eqref{proj1}:
	\begin{equation}\label{swc}
	\forall k\in \N \,\exists n\in\N\, \exists x\in B_N \,\forall i\geq n\, \left(\langle x_0-x, x_i-x\rangle\leq \frac{1}{k+1}\right).
	\end{equation}
	
	In the next section, we will see (cf. \eqref{radius_of_ball} in Lemma~\ref{seq_bounded}) that, if $p\in S$ and {\rm E} satisfies {\rm (Q\ref{q4})}, then for all $n\in \N$
	$$\|x_n-p\|\leq \|x_0-p\|+1+ \sum_{i=0}^{{\rm E}(0)}\|e_i\|,$$
	and in Proposition~\ref{asymp_reg} we will obtain a (monotone) function $\chi_1:\N\to\N$ that is a rate of asymptotic regularity in the following sense:
	\begin{equation}\label{ass_reg}
	\forall k\in \N \,\forall n\geq\chi_{1}(k)\, \left(\|J(x_n)-x_n\|\leq\frac{1}{k+1}\right).
	\end{equation}
	
	With these two facts, we can now give a quantitative form to \eqref{swc}.
	\begin{proposition}\label{q_swc}
		Let $\textsf{A}:X\to 2^X$ a maximal monotone operator on a Hilbert space $X$. Assume that the set $S$ of zeros of $\textsf{A}$ is nonempty and consider the resolvent function $J:=(Id+\textsf{A})^{-1}$. Take $x_0\in X$ and $N\in \N$  a natural number satisfying and $N\geq \max\{2\|x_0-p\|, \|x_0-p\|+1+ \sum_{i=0}^{{\rm E}(0)}\|e_i\|\}$ for some point $p\in S$.\\
		For any $k\in \N$ and monotone function $f:\N \to \N $, there are $n \leq \psi_{N,\chi_1}(k,f)$ and $x\in B_N$ such that
		$$\|J(x)-x \| \leq \frac{1}{f(n)+1} \, \land \, \forall i\geq n 
		\left(\langle x_0-x,x_i-x\rangle \leq \frac{1}{k+1}\right),$$ 
		with $\psi_{N,\chi_1}(k,f):=\chi_1(24N(w_{\hat{f},N}^{(R)}(0)+1)^2)$, where $R:=R(N,k):=4N^4(k+1)^2$ and $\hat{f}(m):=f(\chi_1(m))$.
	\end{proposition}
	
	\begin{proof}
		Let $k\in\N$ and a monotone function $f:\N\to\N$ be given. Notice that $N\geq 2\|x_0-p\|$ for some $p\in S$ and thus, applying Corollary~\ref{proj_J} to the natural number $k$ and to the monotone function $\hat{f}$, we obtain $n_0\leq 24N(w_{\hat{f},N}^{(R)}(0)+1)^2$ and $x\in B_N$ such that
		\begin{equation*}
		\|J(x)-x \| \leq \frac{1}{\hat{f}(n_0)+1} \, \land \, \forall y\in B_N 
		\left(\|J(y)-y\|\leq \frac{1}{n_0+1} \to \langle x_0-x,y-x\rangle \leq \frac{1}{k+1}\right).
		\end{equation*}
		With $n:=\chi_1(n_0)$ -- which by monotonicity is $\leq \psi_{N,\chi_1}(k,f)$ -- we have
		\begin{equation*}
		\|J(x)-x \| \leq \frac{1}{f(n)+1}.
		\end{equation*}
		By \eqref{ass_reg}, for $i\geq n$, we have $\|J(x_i)-x_i\|\leq \frac{1}{n_0+1}$ and, since $(x_n)\subset B_N$, we conclude
		\begin{equation*}
		\langle x_0-x,x_i-x\rangle \leq \frac{1}{k+1}.\qedhere
		\end{equation*}
	\end{proof}

\begin{remark}
	Notice that in the proof above, no instance of sequential weak compactness was used. The theoretical reason for the removal of this principle is fully explained in \cite{ferreira2019removal}. In fact, this elimination can be seen to correspond to an application of \cite[proposition 4.3]{ferreira2019removal} with the parameter function $\varphi$ defined by $\varphi(x,y):=\langle x-x_0, y\rangle$.
\end{remark}

\subsection{Lemmas}

We start with a well-known identity pertaining to resolvent functions.

\begin{lemma}[Resolvent identity]\label{resolvent_identity}
	For all $x\in X$ and $a$, $b \in\R^+$, the following identity holds
	\begin{equation*}
	J_a(x)=J_{b}\left(\frac{b}{a}x+\left(1-\frac{b}{a}\right)J_{a}(x)\right).
	\end{equation*}
\end{lemma}

Lemma \ref{xu_seq_reals_1} above contains the main combinatorial part of the proof of that we wish to analyze and thus, we will need to give it a quantitative version. The particular case of this lemma when $\gamma_n \equiv 0$, was already given a quantitative version by Kohlenbach and Leu{\c s}tean in \cite{kohlenbach2012effective}. Recently with similar arguments, in \cite{leustean2020quantitative} two quantitative versions of Lemma~\ref{xu_seq_reals_1} were shown without that restriction, which we now state.

\begin{lemma}[\cite{leustean2020quantitative}]\label{xu_seq_reals_qt1}
	Let $(s_n)$ be a bounded sequence of non-negative real numbers and $D\in\N$ a positive upper bound on $(s_n)$. Consider sequences of real numbers $(\alpha_n)\subset\,]0,1[$, $(r_n)$ and $(\gamma_n)\subset \R^+_0$ and assume the existence of monotone functions ${\rm A}$, ${\rm R}$, ${\rm G}:\N \to \N$ such that
	\begin{enumerate}[{\rm (i)}]
		\item ${\rm A}$ satisfies {\rm (Q\ref{q2})},
		\item ${\rm R}$ is such that $\forall k\in \N \, \forall n\geq {\rm R}(k) \, \left( r_n \leq \dfrac{1}{k+1} \right)$,
		\item ${\rm G}$ is a Cauchy rate for $(\sum \gamma_n)$.
	\end{enumerate}
	If for all $n\in \N$, $s_{n+1}\leq (1-\alpha_n)s_n+\alpha_nr_n + \gamma_n$,
	then $(s_n)$ converges to zero and
	\begin{equation*}
		\forall k \in \N \,\forall n\geq \theta_1[{\rm A}, {\rm R}, {\rm G}, D](k) \, \left(s_n\leq \dfrac{1}{k+1}\right),
	\end{equation*}
	\noindent where $\theta_1[{\rm A}, {\rm R}, {\rm G}, D](k):={\rm A}\left(M+\lceil \ln(3D(k+1))\rceil\right)+1$, with $M:=\max\{ {\rm R}(3k+2), {\rm G}(3k+2)+1 \}$.\\
\end{lemma}
 
In Lemma~\ref{xu_seq_reals_1}, the fact that $(s_n)$ is bounded follows trivially from the other assumptions. This translates into the easy fact that it is possible to compute a bound $D$ from the remaining data. Namely, a possible value for $D$ is $\lceil \max\{s_0, {\cal R}\} + {\cal G} \rceil$, where ${\cal R}:=\max_{n\leq {\rm R}(0)}\{1, r_n\}$ and ${\cal G}:=1+\sum_{i=0}^{{\rm G}(0)}\gamma_i$ are bounds on the sequences $(r_n)$ and $(\sum \gamma_i)$, respectively.

Consider the condition
\begin{equation}\tag{\rm C2'}\label{c2'}
\forall m\in \N\, \left(\prod_{i= m}^{\infty} (1-\alpha_i)=0\right).
\end{equation}

One can equivalently work with this condition {\rm (\ref{c2'})} instead of considering the condition {\rm (C\ref{c2})}. Hence, it makes sense to also consider a quantitative hypothesis corresponding to {\rm (\ref{c2'})}:
\begin{equation}\tag{\rm Q2'}\label{q2'}
\begin{gathered}
{\rm A'}:\N\times \N\to\N \text{ is a monotone function satisfying}\\
\forall k, m\in \N\, \left( \prod_{i=m}^{{\rm A'}(m,k)}(1-\alpha_i)\leq \frac{1}{k+1}\right),
\end{gathered}
\end{equation}
implying that for each $m\in\N$, $A'(m, \cdot)$ is a rate of convergence towards zero for the sequence $\left(\prod_{i=m}^{n}(1-\alpha_i)\right)_n$. By saying that {\rm A'} is monotone we mean that it is monotone in both variables,
\[
\forall k, k', m, m' \in \N\, \left( k\leq k' \land m\leq m' \to {\rm A'}(m, k)\leq {\rm A'}(m', k')\right).
\]

For particular sequences $(\alpha_n)$, changing between this two conditions may prove to be useful since a function satisfying {\rm (Q\ref{q2})} may have different complexity than a function satisfying {\rm (\ref{q2'})}. An easy example of this is the sequence $(\frac{1}{n+1})$: while we have linear rates of convergence towards zero for $\left(\prod_{i= m}^{n} (1-\frac{1}{i+1})\right)_n$, we only have an exponential rate of divergence for $\left(\sum_{i=0}^{n} \frac{1}{i+1}\right)_n$.\\
Next we state a version of the previous lemma with the condition {\rm (Q\ref{q2})} replaced by the condition \eqref{q2'} -- see \cite[lemma 2.4]{kornlein2015quantitative} and also \cite{leustean2020quantitative}.
\begin{lemma}\label{xu_seq_reals_qt2}
	Let $(s_n)$ be a bounded sequence of non-negative real numbers and $D\in\N$ a positive upper bound on $(s_n)$. Consider sequences of real numbers $(\alpha_n)\subset\, ]0,1[$, $(r_n)\subset \R$ and $(\gamma_n)\subset \R^+_0$ and assume the existence of monotone functions ${\rm A'}: \N\times \N \to \N$ and ${\rm R}$, ${\rm G}:\N \to \N$ such that
	\begin{enumerate}[{\rm (i)}]
		\item ${\rm A'}$ satisfies \eqref{q2'},
		\item ${\rm R}$ is such that $\forall k\in \N \, \forall n\geq {\rm R}(k) \, \left( r_n \leq \frac{1}{k+1} \right)$,
		\item ${\rm G}$ is a Cauchy rate for $(\sum \gamma_n)$.
	\end{enumerate}
	 If for all $n\in \N$, $s_{n+1}\leq (1-\alpha_n)s_n+\alpha_nr_n + \gamma_n$,	then $(s_n)$ converges to zero and
	\begin{equation*}
	\forall k \in \N \,\forall n\geq \theta_2[{\rm A'}, {\rm R},{\rm G}, D](k) \, \left(s_n\leq \dfrac{1}{k+1}\right),
	\end{equation*}
	\noindent where $\theta_2[{\rm A'}, {\rm R}, {\rm G}, D](k):={\rm A'}(M, 3D(k+1)-1)+1$, with $M$ as before.
\end{lemma}

The need for a quantitative version of Lemma~\ref{xu_seq_reals_1} comes from its application in the last step of the proof. At that point, Xu considers for the sequence of real numbers $s_n=\|x_n-P_S(x_0)\|^2$. However in our case, since we are working with approximations to $P_S(x_0)$, the inequality \eqref{xu_lem_main_ass} only holds with $s_n+v_n$ in place of $s_n$, where $(v_n)$ is a sequence of errors. The next quantitative lemmas clarify this situation.

\begin{lemma}\label{final_lem_1}
	Let $(s_n)$ be a bounded sequence of non-negative real numbers and $D\in\N$ a positive upper bound on $(s_n)$. Consider sequences of real numbers $(\alpha_n)\subset\, ]0,1[$, $(r_n)\subset \R$, $(v_n)\subset \R$ and $(\gamma_n)\subset \R^+_0$ and assume the existence of a monotone function ${\rm A}$ satisfying ({\rm Q\ref{q2}}). For natural numbers $k, n$ and $p$ assume
	\[\forall m\in[n,p]\, \left(v_m\leq \frac{1}{4(k+1)(p+1)}\land r_m\leq \frac{1}{4(k+1)}\right),\]
	\[\forall m\in\N\, \left(\sum_{i=n}^{n+m}\gamma_i\leq \frac{1}{4(k+1)}\right)\]
	and 
	\[\forall m\in\N \,\left(s_{m+1}\leq (1-\alpha_m)(s_m+v_m)+\alpha_mr_m+\gamma_m\right).\]
	Then
	\[\forall m\in[\sigma_1(k,n),p]\, \left(s_m\leq \frac{1}{k+1}\right),\]
	with $\sigma_1(k,n):= \sigma_1[{\rm A},D](k,n):={\rm A}\left(n+\lceil \ln(4D(k+1))\rceil\right)+1$.
\end{lemma}

\begin{proof}
	Consider $k,n$ and $p$ such that the premises of the lemma hold. We may assume $p\geq \sigma_1(k,n)$, otherwise the result is trivially true. Since ${\rm A}(m)+1\geq m$, we conclude that $p\geq n$. By induction, we see that for all $m\leq p-n$,
	\begin{equation}\label{final_lem_1_eq1}
	s_{n+m+1}\leq \left(\prod_{i=n}^{n+m}(1-\alpha_i)\right)s_n+\frac{1}{4(k+1)(p+1)}\sum_{j=n}^{n+m}\prod_{i=j}^{n+m}(1-\alpha_i)+\frac{1}{4(k+1)}+\sum_{i=n}^{n+m}\gamma_i.
	\end{equation}
	
	The base case $m=0$ follows from the assumptions of the lemma. For the induction step $m+1\leq p-n$, we have
	\begin{align*}
	&s_{n+m+2}\leq (1-\alpha_{n+m+1})(s_{n+m+1}+v_{n+m+1})+\alpha_{n+m+1}r_{n+m+1}+\gamma_{n+m+1}\\
	&\;\leq (1-\alpha_{n+m+1})\left[\left(\prod_{i=n}^{n+m}(1-\alpha_i)\right)s_n+\frac{1}{4(k+1)(p+1)}\sum_{j=n}^{n+m}\prod_{i=j}^{n+m}(1-\alpha_i)+\frac{1}{4(k+1)}+\sum_{i=n}^{n+m}\gamma_i\right]\\
	&\qquad +(1-\alpha_{n+m+1})v_{n+m+1}+\alpha_{n+m+1}\frac{1}{4(k+1)}+ \gamma_{n+m+1}\\
	&\;\leq \left(\prod_{i=n}^{n+m+1}(1-\alpha_i)\right)s_n + \frac{1}{4(k+1)(p+1)}\sum_{j=n}^{n+m+1}\prod_{i=j}^{n+m+1}(1-\alpha_i) + \frac{1}{4(k+1)}+\sum_{i=n}^{n+m+1}\gamma_i,
	\end{align*}
	using the induction hypothesis and the fact that, since $n+m+1\in[n,p]$, $r_{n+m+1}\leq \frac{1}{4(k+1)}$. This concludes the induction.\\
	
	Since for $m\leq p-n$, we have $\sum_{j=n}^{n+m}\prod_{i=j}^{n+m}(1-\alpha_i)\leq m+1\leq p+1$, we conclude
	\begin{equation*}
	\frac{1}{4(k+1)(p+1)}\sum_{j=n}^{n+m}\prod_{i=j}^{n+m}(1-\alpha_i)\leq \frac{1}{4(k+1)}.
	\end{equation*}
	Since $\sum_{i=n}^{n+m}\gamma_i\leq \frac{1}{4(k+1)}$, by \eqref{final_lem_1_eq1}, we conclude that for all $m\leq p-n$
	\begin{equation}\label{final_lem_1_eq2}
	s_{n+m+1}\leq D\left(\prod_{i=n}^{n+m}(1-\alpha_i)\right)+\frac{3}{4(k+1)}.
	\end{equation}
	
	Define the natural number $K:={\rm A}\left(n+\lceil \ln(4D(k+1))\rceil\right)-n$. For $m\geq K$, we have
	\begin{equation*}
	\sum_{i=n}^{n+m}\alpha_i\geq \sum_{i=n}^{n+K}\alpha_i=\sum_{i=0}^{{\rm A}\left(n+\lceil \ln(4D(k+1))\rceil\right)}\alpha_i - \sum_{i=0}^{n-1}\alpha_i\geq n+\ln(4D(k+1))- \sum_{i=0}^{n-1}\alpha_i\geq \ln(4D(k+1)).
	\end{equation*}
	Using the fact that for $x\in\R_0^+$, $1-x\leq \exp(-x)$, we obtain for all $m\geq K$,
	\begin{equation}\label{final_lem_1_eq3}
	D\prod_{i=n}^{n+m}(1-\alpha_i)\leq D\exp\left(-\sum_{i=n}^{n+m}\alpha_i\right)\leq \frac{1}{4(k+1)}.
	\end{equation}
	Finally, from \eqref{final_lem_1_eq2} and \eqref{final_lem_1_eq3} together, for $m\in[K,p-n]$, $s_{n+m+1}\leq \frac{1}{k+1}$
	and thus, for $m\in[n+K+1, p]=[\sigma_1(k,n), p]$, we have $s_m\leq \frac{1}{k+1}$.
\end{proof}

\begin{remark}
	Notice that the function $\sigma_1$, in theory, could additionally depend on the value of $p$ and the fact that it doesn't is crucial for the analysis (see the proof of Theorem~\ref{main1}). Furthermore, if the assumptions would hold for arbitrary $p\in\N$, then $v_n\equiv 0$ and one would conclude $\forall m\geq \sigma_1(k,n)\, \left(s_m\leq \frac{1}{k+1}\right)$, which also explains the obvious connection with the function $\theta_1$ in Lemma~\ref{xu_seq_reals_qt1}.
\end{remark}
In a similar way, we can again conclude the result but with the function ${\rm A'}$ satisfying {\rm (\ref{q2'})} instead.
\begin{lemma}\label{final_lem_2}
	Let $(s_n)$ be a bounded sequence of non-negative real numbers and $D\in\N$ a positive upper bound on $(s_n)$. Consider sequences of real numbers $(\alpha_n)\subset\, ]0,1[$, $(r_n)\subset \R$, $(v_n)\subset \R$ and $(\gamma_n)\subset \R^+_0$ and assume the existence of a monotone function ${\rm A'}:\N \times \N \to \N$ satisfying the condition {\rm (\ref{q2'})}. For natural numbers $k, n$ and $p$ assume
	\[\forall m\in[n,p]\, \left(v_m\leq \frac{1}{4(k+1)(p+1)}\land r_m\leq \frac{1}{4(k+1)}\right),\]
	\[\forall m\in\N\, \left(\sum_{i=n}^{n+m}\gamma_i\leq \frac{1}{4(k+1)}\right)\]
	and
	\[\forall m\in\N\, \left(s_{m+1}\leq (1-\alpha_m)(s_m+v_m)+\alpha_mr_m+\gamma_m\right).\]
	Then
	\[\forall m\in[\sigma_2(k,n),p]\, \left(s_m\leq \frac{1}{k+1}\right),\]
	with $\sigma_2(k,n):=\sigma_2[{\rm A'}, D](k,n):={\rm A'}\left(n,4D(k+1)-1\right)+1$.
\end{lemma}

\begin{proof}
	Following the proof of the previous lemma, we conclude that for all $m\leq p-n$
	\begin{equation*}
	s_{n+m+1}\leq D\left(\prod_{i=n}^{n+m}(1-\alpha_i)\right)+\frac{3}{4(k+1)}
	\end{equation*}
	
	Define the natural number $K:={\rm A'}\left(n,4D(k+1)-1\right)-n$. By {\rm (\ref{q2'})}, we have for all $m\geq K$,
	\begin{equation*}
	\prod_{i=n}^{n+m}(1-\alpha_i)\leq \prod_{i=n}^{n+K}(1-\alpha_i)\leq \frac{1}{4D(k+1)}.
	\end{equation*}
		
	This shows that for $m\in[n+K+1, p]=[\sigma_2(k,n), p]$, $s_m\leq \frac{1}{k+1}$.
\end{proof}

\section{Main results}

We start by computing a bound on the sequence $(x_n)$.

\begin{lemma}\label{seq_bounded}
	Let $p$ be an arbitrary point in $S$. Then, for all $n\in \N$, $\|x_n-p\|\leq \|x_0-p\|+\sum_{i=0}^{n-1}\|e_i\|$.\\
	Let ${\rm E}:\N\to\N$ be a monotone function satisfying \rm(Q\ref{q4}), i.e. a \emph{Cauchy rate} for $\left(\sum_{i=0}^{n} \|e_i\|\right)$, and ${\cal E}\in\N$ a natural number satisfying ${\cal E}\geq 1+ \sum_{i=0}^{{\rm E}(0)}\|e_i\|$. We have for all $n\in\N$
	\begin{align}
&\|x_n-p\|\leq \|x_0-p\| + \mathcal{E}\label{bound1}\\
&\|x_n\|\leq  \|x_0-p\|+\|p\| + \mathcal{E}\label{bound2}\\
&\|x_n-x_0\|\leq 2\|x_0-p\| +\mathcal{E}\label{bound3}
	\end{align}
\end{lemma}

\begin{proof}
	Since $p$ is a point in $S$, it is a fixed point for resolvent functions. For any $n\in\N$, we have
	\begin{align}
	\|x_{n+1}-p\|&\leq \|\alpha_nx_0+(1-\alpha_n)(J_{\beta_n}(x_n)+e_n)-p\|\nonumber\\
	&\leq \alpha_n\|x_0-p\|+(1-\alpha_n)\|J_{\beta_n}(x_n)-p\|+\|e_n\|\nonumber\\
	&\leq \alpha_n\|x_0-p\|+(1-\alpha_n)\|x_n-p\|+\|e_n\|\label{ind_bounded}
	\end{align}
	
	By induction on $n\in\N$, we see
	\begin{equation}\label{radius_of_ball}
	\forall n\in \N\, \left( \|x_n-p\|\leq \|x_0-p\|+\sum_{i=0}^{n-1}\|e_i\| \right).
	\end{equation}
	The case $n=0$ is trivial. For the induction step $n+1$, use \eqref{ind_bounded} and the induction hypothesis.
	\begin{align*}
	\|x_{n+1}-p\|&\leq \alpha_n\|x_0-p\|+(1-\alpha_n)\|x_n-p\|+\|e_n\| \\
	&\leq \alpha_n\|x_0-p\|+(1-\alpha_n)\left( \|x_0-p\|+\sum_{i=0}^{n-1}\|e_i\| \right)+\|e_n\|\\
	&\leq \|x_0-p\|+\sum_{i=0}^{n}\|e_i\|
	\end{align*}
	
	Now the remaining inequalities follow easily. Since ${\rm E}$ is a Cauchy rate for $\left(\sum_{i=0}^{n} \|e_i\|\right)$ and ${\cal E} \geq 1+ \sum_{i=0}^{{\rm E}(0)}\|e_i\|$, we get that for all $n\in\N$, $\sum_{i=0}^{n}\|e_i\| \leq {\cal E}$ and the inequality \eqref{bound1} follows. Since $\|x_n\|\leq \|x_n-p\|+\|p\|$, \eqref{bound2} follows from \eqref{bound1}. Similarly for \eqref{bound3}.
\end{proof}

The following rate of convergence is easily derived from the original proof of Theorem~\ref{original_xu}.
\begin{lemma}\label{asymp_reg0}
	Consider monotone functions ${\rm a}$, ${\rm E}:\N\to\N$ satisfying {\rm (Q\ref{q1})} and {\rm (Q\ref{q4})}, respectively. Let ${\cal E}, {\cal D}\in\N$ be natural numbers satisfying ${\cal E}\geq 1+ \sum_{i=0}^{{\rm E}(0)}\|e_i\|$ and ${\cal D} \geq \|x_0-p\|$, for some $p\in S$. For every $k\in\N$, define $\xi[{\rm a}, {\rm E}, {\cal E}, {\cal D}](k):=\max\{{\rm a}(2\left(2{\cal D}+{\cal E}\right)(k+1)-1), {\rm E}(2k+1)+1\}$. Then
	\begin{equation*}
	\forall k\in \N \,\forall n\geq \xi[{\rm a}, {\rm E}, {\cal E}, {\cal D}](k) \, \left(\|x_{n+1}-J_{\beta_n}(x_n)\|\leq \frac{1}{k+1}\right).
	\end{equation*}
\end{lemma}

\begin{proof}
	Let $k\in\N$ be given and consider $n\geq \xi(k):=\xi[{\rm a}, {\rm E}, {\cal E}, {\cal D}](k)$. Since $n\geq {\rm E}(2k+1)+1$, from condition {\rm (Q\ref{q4})}, we have in particular $\|e_n\|\leq \frac{1}{2(k+1)}$. From Lemma~\ref{seq_bounded}, we have $\|x_n-p\|\leq {\cal D}+{\cal E}$, for a certain $p\in S$. Using the fact that $J_{\beta_n}$ is nonexpansive and $p$ is a fixed point of $J_{\beta_n}$, we get
	\begin{align*}
	\|x_{n+1}-J_{\beta_n}(x_n)\|&=\|\alpha_nx_0+(1-\alpha_n)(J_{\beta_n}(x_n)+e_n)-J_{\beta_n}(x_n)\|\\
	&\leq \alpha_n\|x_0-J_{\beta_n}(x_n)\|+\|e_n\|\\
	&\leq \alpha_n(\|x_0-p\|+\|p-J_{\beta_n}(x_n)\|)+\|e_n\| \\
	&\leq \alpha_n(\|x_0-p\|+\|p-x_n\|)+\|e_n\|\\
	&\leq \alpha_n\left(2{\cal D}+{\cal E}\right)+\|e_n\|\\
	&\leq\frac{2{\cal D}+{\cal E}}{2\left(2{\cal D}+{\cal E}\right)(k+1)}+\frac{1}{2(k+1)}= \frac{1}{k+1},
	\end{align*}
	using in the last step the fact that $n\geq {\rm a}(2\left(2{\cal D}+{\cal E}\right)(k+1)-1)$ and the condition {\rm (Q\ref{q1})}.
\end{proof}

Next we compute a rate of asymptotic regularity for the sequence $(x_n)$ in relation to a resolvent function $J_{\gamma}$, for an arbitrary positive real number $\gamma$.
\begin{proposition}\label{asymp_reg}
	Consider a real number $\gamma >0$ and monotone functions ${\rm a}, {\rm B}, {\rm E}:\N\to\N$ satisfying {\rm (Q\ref{q1})}, {\rm (Q\ref{q3})} and {\rm (Q\ref{q4})}, respectively. Let ${\cal E}$, ${\cal D}, \ell\in \N$ be natural numbers satisfying ${\cal E}\geq 1+ \sum_{i=0}^{{\rm E}(0)}\|e_i\|$, ${\cal D} \geq \|x_0-p\|$, for some $p\in S$ and $\ell \geq \gamma$.
	For each $k\in\N$, define
	\begin{equation*}
	\chi_{\ell}(k):=\chi_{\ell}[{\rm a}, {\rm B}, {\rm E}, {\cal E}, {\cal D}](k):=\max\{\xi[{\rm a}, {\rm E}, {\cal E}, {\cal D}](4k+3),{\rm B}(8({\cal D}+{\cal E})(k+1)\ell -1) \}+1
	\end{equation*}
	where $\xi$ is as in Lemma \ref{asymp_reg0}. Then
	\begin{equation*}
	\forall k\in \N \,\forall n \geq \chi_{\ell}(k)\, \left(\|x_n-J_{\gamma}(x_n)\|\leq \frac{1}{k+1}\right).
	\end{equation*}
\end{proposition}

\begin{proof}
	First notice that from the definition of $\xi$ and by the monotonicity of the functions ${\rm a}$ and ${\rm E}$, we have
	\begin{equation*}
	\chi_{\ell}(k)-1\geq {\rm a}(4(2{\cal D}+{\cal E})(k+1)-1)\, \text{ and }\, \chi_{\ell}(k)-1\geq {\rm E}(4k+3)+1.
	\end{equation*}
	
	For $n+1\geq \chi_{\ell}(k)$, using Lemma~\ref{resolvent_identity} (resolvent identity), we have
	\begin{align}
	\left\|x_{n+1}-J_{\gamma}(x_{n+1})\right\|&\leq \alpha_n\left\|x_0-J_{\gamma}(x_{n+1})\right\|+\left\|J_{\beta_n}(x_n)-J_{\gamma}(x_{n+1})\right\|+\|e_n\| \nonumber\\
	&\,\leq (2{\cal D}+{\cal E})\alpha_n + \left\|J_{\gamma}\left( \frac{\gamma}{\beta_n}x_n +(1-\frac{\gamma}{\beta_n})J_{\beta_n}(x_n)\right)-J_{\gamma}(x_{n+1})\right\|+\|e_n\| \nonumber\\
	&\leq (2{\cal D}+{\cal E})\alpha_n +\left\|\frac{\gamma}{\beta_n}x_n+(1-\frac{\gamma}{\beta_n})J_{\beta_n}(x_n)-x_{n+1}\right\|+\|e_n\|\nonumber\\
	&\leq (2{\cal D}+{\cal E})\alpha_n+ \frac{\gamma}{\beta_n}\|x_n-x_{n+1}\|+\left|1-\frac{\gamma}{\beta_n}\right|\|J_{\beta_n}(x_n)-x_{n+1}\|+\|e_n\|\nonumber \\
	&\leq (2{\cal D}+{\cal E})\alpha_n + \frac{\gamma}{\beta_n}2({\cal D}+{\cal E})+\left|1-\frac{\gamma}{\beta_n}\right|\|J_{\beta_n}(x_n)-x_{n+1}\|+\|e_n\|\label{asymp_reg_eq}	
	\end{align}
	
	Since $n\geq \chi_{\ell}(k)-1\geq {\rm a}(4(2{\cal D}+{\cal E})(k+1)-1)$, we get
	\begin{equation*}
	(2{\cal D}+{\cal E})\alpha_n \leq \frac{2{\cal D}+{\cal E}}{4(2{\cal D}+{\cal E})(k+1)}=\frac{1}{4(k+1)}.
	\end{equation*}
	Similarly, from $n\geq {\rm E}(4k+3)+1$ and {\rm (Q\ref{q4})}, we have in particular $\|e_n\|\leq \dfrac{1}{4(k+1)}$. Notice that from {\rm (Q\ref{q3})} and $n\geq {\rm B}(8({\cal D}+{\cal E})(k+1)\ell -1)$, we conclude $\dfrac{1}{\beta_n}\leq \dfrac{1}{8({\cal D}+{\cal E})(k+1)\ell}$ and also $\left|1-\dfrac{\gamma}{\beta_n}\right|\leq 1$.\\
	Finally, using \eqref{asymp_reg_eq} and the fact that $n\geq \xi[{\rm a}, {\rm E}, {\cal E}, {\cal D}](4k+3)$, we conclude
	\begin{equation*}
\left\|x_{n+1}-J_{\gamma}(x_{n+1})\right\|\leq\frac{1}{4(k+1)}+ \frac{2({\cal D}+{\cal E})\gamma}{8({\cal D}+{\cal E})(k+1)\ell}+ \frac{1}{4(k+1)}+ \frac{1}{4(k+1)}\leq\frac{1}{k+1},
	\end{equation*}
	which concludes the proof.
\end{proof}

Before, we looked at the projection onto $S$ by referring the fixed-point set of $J_1$. There is no problem in focusing on that particular set since all fixed-point sets of resolvent functions associated with the maximal monotone operator $\textsf{A}$ coincide (with $S$). Bellow we give a quantitative version of the statement that any two resolvent functions have the same fixed points.
\begin{lemma}\label{same_fix_set_ab}
	For all $\alpha$, $\beta\in \R^+$, $k\in \N$ and $x\in X$,
	\begin{equation*}
	\|J_{\alpha}(x)-x\|\leq \frac{1}{\max\{2-\frac{\beta}{\alpha}, \frac{\beta}{\alpha}\}(k+1)}\,\to\, \|J_{\beta}(x)-x\|\leq \frac{1}{k+1}.
	\end{equation*}
\end{lemma}

\begin{proof}
By the resolvent identity, we have
\begin{align*}
\|J_{\beta}(x)-x\|&\leq \|J_{\beta}(x)-J_{\alpha}(x)\|+\|J_{\alpha}(x)-x\|\\
&\leq \left\|J_{\beta}(x)-J_{\beta}\left(\frac{\beta}{\alpha}x+\left(1-\frac{\beta}{\alpha}\right)J_{\alpha}(x)\right)\right\|+\|J_{\alpha}(x)-x\|\\
&\leq \left\|\left(1-\frac{\beta}{\alpha}\right)(x-J_{\alpha}(x))\right\|+\|J_{\alpha}(x)-x\|=\left(1+\left|1-\frac{\beta}{\alpha}\right|\right)\|J_{\alpha}(x)-x\|. 
\end{align*}

The result now follows from noticing that $(1+|1-\frac{\beta}{\alpha}|)\leq \max\{2-\frac{\beta}{\alpha}, \frac{\beta}{\alpha}\}$. Indeed, we have
\begin{align*}
&\beta\leq \alpha \, \to\, \|J_{\beta}(x)-x\|\leq \left(2-\frac{\beta}{\alpha}\right)\|J_{\alpha}(x)-x\|\\
&\beta > \alpha \, \to \, \|J_{\beta}(x)-x\|\leq \frac{\beta}{\alpha}\|J_{\alpha}(x)-x\|
\end{align*}
\end{proof}

We can particularize this result to the resolvent functions $J_{\beta_n}$ and $J\,(:=J_1)$.
\begin{lemma}\label{same_fix_set}
	Consider ${\rm b}:\N\to\N$ to be a monotone function satisfying
	\begin{equation}\label{upper_pointwise_bound}\tag{Q3'}
	\forall n\in\N \, \left( \beta_n\leq {\rm b}(n)\right).
	\end{equation}
	For all $k$, $n\in \N$ and $x\in X$,
	\begin{equation*}
	\|J(x)-x\|\leq \frac{1}{\delta_{\rm b}(k,n)+1}\, \to \forall i\leq n\, \left(\|J_{\beta_{i}}(x)-x\|\leq \frac{1}{k+1}\right),
	\end{equation*}
	where $\delta_{\rm b}(k,n):=\max\{2, {\rm b}(n)\}(k+1)-1$.
\end{lemma}

\begin{proof}
	The result follows from the previous lemma. Consider $n \in \N$. Since {\rm b} is a monotone function, notice that for all $i\leq n$, we have $\beta_i\leq {\rm b}(n)$ which implies $\max\{2, {\rm b}(n)\}\geq \max\{2-\beta_i, \beta_i\}$.
\end{proof}

From Proposition \ref{q_swc} and Lemma \ref{same_fix_set}, we derive the following result.
\begin{proposition}\label{psi+delta}
	Consider monotone functions ${\rm a}, {\rm b}, {\rm B}, {\rm E}:\N\to\N$ satisfying {\rm (Q\ref{q1})}, \eqref{upper_pointwise_bound}, {\rm (Q\ref{q3})} and {\rm (Q\ref{q4})}, respectively. Let ${\cal E}$, ${\cal D}\in \N$ be natural numbers satisfying ${\cal E}\geq 1+ \sum_{i=0}^{{\rm E}(0)}\|e_i\|$ and ${\cal D} \geq \|x_0-p\|$ for some $p\in S$, and write $N:=\max\{2{\cal D}, {\cal D}+{\cal E}\}$. For every $k \in \N$ and any monotone function $f:\N\to\N$, there exists $n\leq \Psi(k,f)$ and $x\in B_{N}$ such that
	\begin{equation*}
	\forall i\in [n,fn]\, \left(\|J_{\beta_i}(x)-x\| \leq \dfrac{1}{f(n)+1}  \,\land \,\langle x_0-x,x_i-x\rangle \leq \frac{1}{k+1}\right),
	\end{equation*}	
	with $\Psi(k,f):=\Psi[{\rm a}, {\rm b}, {\rm B}, {\rm E}, {\cal E}, {\cal D}](k,f):=\psi_{N,\chi_1}(k,\nu_f)$, where $\chi_1=\chi_1[{\rm a}, {\rm B}, {\rm E}, {\cal E}, {\cal D}]$, $\psi_{N,\chi_1}$ as in Proposition~\ref{q_swc} and $\nu_f(m)=\delta_{\rm b}\left( f(m),f(m)\right)$.
\end{proposition}

\begin{proof}
	Let $k\in \N$ and a monotone function $f:\N\to\N$ be given and consider the monotone function $\nu_f$. Notice that by the definition of the function $\delta_{\rm b}$, for all $m\in\N$, $\nu_f(m)\geq f(m)$. By Proposition \ref{q_swc}, we see that there are $n\leq \psi_{N,\chi_1}(k,\nu_f)=:\Psi(k,f)$ and $x\in B_N$ such that
	\begin{equation*}
	\|J(x)-x\|\leq \frac{1}{\nu_f(n)+1} \,\text{ and }\, \forall i\geq n\, \left( \langle x_0-x,x_i-x\rangle \leq \frac{1}{k+1}\right).
	\end{equation*}
	By Lemma~\ref{same_fix_set}, we get that for $i\leq f(n)$, $\|J_{\beta_i}(x)-x\|\leq \frac{1}{f(n)+1}$ and the result follows.
\end{proof}

We are now ready to give the quantitative version of Theorem \ref{original_xu}.
\begin{theorem}\label{main1}
	Let $\textsf{A}$ be a maximal monotone operator and $S$ the set of zeros of $\textsf{A}$ which is assumed to be nonempty. Consider sequences $(\alpha_n) \subset \,]0,1[$, $(\beta_n)\subset \R^+$ and $(e_n)\subset X$. With $x_0\in X$, let $(x_n)$ be the corresponding Halpern type proximal point iteration recursively defined by \eqref{hppa}.\\
	Assume the existence of monotone functions {\rm a}, {\rm A}, {\rm b}, {\rm B}, ${\rm E}:\N \to \N$ such that the conditions {\rm (Q\ref{q1})}, {\rm (Q\ref{q2})}, \eqref{upper_pointwise_bound}, {\rm (Q\ref{q3})} and {\rm (Q\ref{q4})} are satisfied. Let ${\cal E}$, ${\cal D}\in \N$ be natural numbers satisfying ${\cal E}\geq 1+ \sum_{i=0}^{{\rm E}(0)}\|e_i\|$ and ${\cal D} \geq \|x_0-p\|$ for some $p\in S$, and write $N:=\max\{2{\cal D}, {\cal D}+{\cal E}\}$. Then for all $k\in\N$ and any (monotone) function $f:\N\to\N$,
	\begin{equation*}
	\exists n\leq \Phi_1(k,f)\, \forall i,j\in [n,f(n)]\, \left(\|x_i-x_j\|\leq \frac{1}{k+1}\right),
	\end{equation*}
	where $\Phi_1(k,f):=\Phi_1[{\rm a}, {\rm A}, {\rm b}, {\rm B}, {\rm E}, {\cal E}, {\cal D}](k,f)$ with
	\begin{enumerate}[\indent{}]
	\item $\Phi_1[{\rm a}, {\rm A}, {\rm b}, {\rm B}, {\rm E}, {\cal E}, {\cal D}](k,f):=\sigma_1(\tilde{k}, g(\Delta))$,
	\item $\sigma_1:=\sigma_1[{\rm A}, 4N^2]$ is as in Lemma~\ref{final_lem_1},
	\item $\tilde{k}:=4(k+1)^2-1$,
	\item $g(m):=g[{\rm E}, N, k](m):=\max\{ m,{\rm E}\left( 16(1+4N)(k+1)^2-1\right)+1\}$,
	\item $\Delta:=\Psi(32(k+1)^2-1, h_f)$,
	\item $\Psi:=\Psi[{\rm a}, {\rm b}, {\rm B}, {\rm E}, {\cal E}, {\cal D}]$ is as in Proposition~\ref{psi+delta},
	\item $h_f(m):=h[{\rm A}, {\rm E}, N, k, f](m):=(1+4N)\left( 16(k+1)^2\left( f(\sigma_1(\tilde{k}, g(m)))+1 \right) +1 \right)-1$.
	\end{enumerate}
\end{theorem}

\begin{proof}
	Under the hypothesis of the theorem, let $k\in \N$ and a monotone function $f$ be given. We divide the proof in a series of claims and our main goal is to apply Lemma~\ref{final_lem_1} to $4(k+1)^2-1$ and with $p=f(\sigma_1(\tilde{k}, g(n)))$, for a certain value of $n$.\\
	
	\textbf{Claim 1:} There are $n\leq \Delta$ and $x\in B_N$ such that for all $i\in [n,f(\sigma_1(\tilde{k}, g(n)))]$
	\begin{equation*}
	\begin{gathered} 
	\left\|J_{\beta_i}(x)-x\right\|\left(\left\|J_{\beta_i}(x)-x\right\| + 2\|x_i-x\|\right) \leq \dfrac{1}{16(k+1)^2(f(\sigma_1(\tilde{k}, g(n)))+1)+1}\\
	\text{and }\, \,2\langle x_0-x,x_{i+1}-x\rangle \leq \frac{1}{16(k+1)^2}.
	\end{gathered}
	\end{equation*}	
	
	\textbf{Proof of Claim 1:} By Proposition~\ref{psi+delta}, there are $n\leq \Psi(32(k+1)^2-1,h_f)=:\Delta$ and $x\in B_N$ such that for all $i\in [n,h_f(n)]$
	\begin{equation}\label{claim_eq1}
	\|J_{\beta_i}(x)-x\|\leq \frac{1}{h_f(n)+1}\, \text{ and }\,\langle x_0-x,x_i-x\rangle \leq \frac{1}{32(k+1)^2}.
	\end{equation}
	
	If $i\in[n,f(\sigma_1(\tilde{k}, g(n)))]$, then $i+1\in[n,f(\sigma_1(\tilde{k}, g(n)))+1]\subset [n,h_f(n)]$. Hence, for $i\in[n,f(\sigma_1(\tilde{k}, g(n)))]$, by \eqref{claim_eq1},
	\begin{equation*}
	2\langle x_0-x, x_{i+1}-x \rangle \leq \frac{1}{16(k+1)^2}.
	\end{equation*}
	
	For $i\in [n,f(\sigma_1(\tilde{k}, g(n)))]$, and noticing that $2N\geq \|x_i-x\|$ and $\|J_{\beta_i}(x)-x\|\leq 1$, we have
	\begin{align*}
	&\|J_{\beta_i}(x)-x\|\left(\|J_{\beta_i}(x)-x\| + 2\|x_i-x\|\right)\leq \|J_{\beta_i}(x)-x\|\left(1 + 4N\right)\leq\\
	&\quad\leq \frac{1+4N}{h_f(n)+1}= \frac{1+4N}{(1+4N)\left(16(k+1)^2(f(\sigma_1(\tilde{k}, g(n)))+1)+1\right)}=\frac{1}{16(k+1)^2(f(\sigma_1(\tilde{k}, g(n)))+1)+1},
	\end{align*}
	which concludes the proof of the claim.\hfill$\blacksquare$\\
	
	Let $n_1\leq \Delta$ and $\tilde{x}\in B_N$ be as in Claim 1. Define the following sequences:
	\begin{enumerate}[\indent{}]
	\item $s_m:=\|x_m-\tilde{x}\|^2$,
	\item $v_m:=\left\|J_{\beta_m}(\tilde{x})-\tilde{x}\right\|\left(\left\|J_{\beta_m}(\tilde{x})-\tilde{x}\right\| + 2\|x_m-\tilde{x}\|\right)$,
	\item $r_m:=2\langle x_0-\tilde{x},x_{m+1}-\tilde{x}\rangle$,
	\item $\gamma_m:=\|e_m\|(\|e_m\|+2\|J_{\beta_m}(x_m)-\tilde{x}\|)$.
	\end{enumerate}
	
	We point out that $4N^2$ is a positive upper bound on the sequence $(s_m)$.\\
	
	\textbf{Claim 2:} Consider the natural number $n_2:=g(n_1)$. Then
	\begin{equation*}
	\forall m\in \N \,\left(\sum_{i=n_2}^{n_2+m}\gamma_i\leq \frac{1}{16(k+1)^2}\right).
	\end{equation*}
	\textbf{Proof of Claim 2:} We recall that $n_2:=\max\{n_1, {\rm E}(16(1+4N)(k+1)^2-1)+1\}$. Notice that for all $i\in\N$, $\|J_{\beta_i}(x_i)-\tilde{x}\|\leq 2N$. Also, by {\rm (Q\ref{q4})}, for $i\geq n_2$ we have $i\geq {\rm E}(16(1+4N)(k+1)^2-1)+1$ and so $\|e_i\|\leq \frac{1}{16(1+4N)(k+1)^2}\leq 1$. Therefore, for all $m\in\N$,
	\begin{align*}
	\sum_{i=n_2}^{n_2+m}\gamma_i&= \sum_{i=n_2}^{n_2+m} \|e_i\|(\|e_i\|+2\|J_{\beta_i}(x_i)-\tilde{x}\|) \leq\\
	&\leq (1+4N)\sum_{i=n_2}^{n_2+m} \|e_i\|\leq \frac{1+4N}{16(1+4N)(k+1)^2}=\frac{1}{16(k+1)^2},
	\end{align*}
	which concludes the proof of Claim 2.\hfill$\blacksquare$\\
	
	\textbf{Claim 3:} For any $m\in \N$,
	\begin{equation*}
	s_{m+1}\leq (1-\alpha_m)\left(s_m+v_m\right) + \alpha_mr_m+ \gamma_m.
	\end{equation*}
	
	\textbf{Proof of Claim 3:} This inequality is obtained by using the subdifferential inequality,
	\begin{equation*}
	\|x+y\|^2\leq \|x\|^2+2\langle y,x+y\rangle,
	\end{equation*}
	in the following way,
	\begin{align*}
	s_{m+1}&=\|x_{m+1}-\tilde{x}\|^2=\|(1-\alpha_m)(J_{\beta_m}(x_m)+e_m-\tilde{x})+\alpha_m(x_0-\tilde{x})\|^2\\
	& \leq (1-\alpha_m)^2\|J_{\beta_m}(x_m)+e_m-\tilde{x}\|^2+ \alpha_m(2\langle x_0-\tilde{x}, x_{m+1}-\tilde{x}\rangle)\\
	& \leq (1-\alpha_m)\left( \|J_{\beta_m}(x_m)-\tilde{x}\|+\|e_m\| \right)^2+\alpha_mr_m\\
	&\leq (1-\alpha_m)\|J_{\beta_m}(x_m)-\tilde{x}\|^2+\alpha_mr_m+\|e_m\|(\|e_m\|+2\|J_{\beta_m}(x_m)-\tilde{x}\|) \\
	&\leq (1-\alpha_m)\left( \|J_{\beta_m}(x_m)-J_{\beta_m}(\tilde{x})\|+\|J_{\beta_m}(\tilde{x})-\tilde{x}\| \right)^2+\alpha_mr_m+\gamma_m\\
	&\leq (1-\alpha_m)\left(\|x_m-\tilde{x}\|^2+ \|J_{\beta_m}(\tilde{x})-\tilde{x}\|\left(\|J_{\beta_m}(\tilde{x})-\tilde{x}\|+2\|x_m-\tilde{x}\|\right)\right) + \alpha_mr_m+\gamma_m\\
	&=(1-\alpha_m)(s_m+v_m)+\alpha_mr_m+\gamma_m.
	\end{align*}
	\hfill$\blacksquare$\\
	
	Since $n_1\leq n_2$, from Claim 1 we also have
	\begin{equation*}
	\forall i\in [n_2, f(\sigma_1(\tilde{k}, n_2))]\,\left(  v_n\leq \dfrac{1}{16(k+1)^2(f(n)+1)+1}  \,\land \,r_n \leq \frac{1}{16(k+1)^2}\right).
	\end{equation*}
	Hence we are in the conditions of Lemma~\ref{final_lem_1} for $k=\tilde{k}$, $n=n_2$ and $p=f(\sigma_1(\tilde{k}, n_2))$. Therefore, with $n_3:=\sigma_1(\tilde{k}, n_2)\leq \sigma_1(\tilde{k}, g(\Delta))=:\Phi_1(k,f)$, we conclude,
	\begin{equation*}
	\forall i\in [n_3, f(n_3)]\, \left( \|x_i-\tilde{x}\|^2\leq \frac{1}{4(k+1)^2}\right).
	\end{equation*}
	This implies $\|x_i-\tilde{x}\|\leq \frac{1}{2(k+1)}$ and for $i, j\in[n_3, f(n_3)]$,
	\begin{equation*}
	\|x_i-x_j\|\leq \|x_i-\tilde{x}\|+\|x_j-\tilde{x}\|\leq \frac{1}{k+1},
	\end{equation*}
	concluding the proof.
\end{proof}

	In a similar way, using Lemma~\ref{final_lem_2} we can conclude a bound on the metastability for $(x_n)$ when the function {\rm A} satisfies instead the condition \eqref{q2'}.
	\begin{theorem}\label{main2}
		Under the hypothesis of Theorem~\ref{main1}, except that, instead of the function ${\rm A}$, we now consider a function ${\rm A'}:\N\times \N \to \N$ satisfying the condition \eqref{q2'}.
		 For all $k\in\N$ and any (monotone) function $f:\N\to\N$,
		\begin{equation*}
		\exists n\leq \Phi_2(k,f)\, \forall i,j\in [n,f(n)]\, \left(\|x_i-x_j\|\leq \frac{1}{k+1}\right),
		\end{equation*}
		where $\Phi_2(k,f):=\Phi_2[{\rm a}, {\rm A'}, {\rm b}, {\rm B}, {\rm E}, {\cal E}, {\cal D}](k,f)$ with
		\begin{enumerate}[\indent{}]
			\item $\Phi_2[{\rm a}, {\rm A'}, {\rm b}, {\rm B}, {\rm E}, {\cal E}, {\cal D}, ](k,f):=\sigma_2(\tilde{k}, g(\Delta))$,
			\item $\sigma_2:=\sigma_2[{\rm A'}, 4N^2]$ is as in Lemma~\ref{final_lem_2},
			\item $\tilde{k}:=4(k+1)^2-1$,
			\item $g(m):=g[{\rm E}, N, k](m):=\max\{ m,{\rm E}\left( 16(1+4N)(k+1)^2-1\right)+1\}$,
			\item $\Delta:=\Psi(32(k+1)^2-1, \bar{h}_f)$,
			\item $\Psi:=\Psi[{\rm a}, {\rm b}, {\rm B}, {\rm E}, {\cal E}, {\cal D}]$ is as in Proposition~\ref{psi+delta},
			\item $\bar{h}_f(m):=h[{\rm A}, {\rm E}, N, k, f](m):=(1+4N)\left( 16(k+1)^2\left( f(\sigma_2(\tilde{k}, g(m)))+1 \right) +1 \right)-1$.
		\end{enumerate}
	\end{theorem}
	
	\begin{proof}
		The proof follows the same steps as before, now using the functions $\sigma_2$ and $\bar{h}_f$, and applying Lemma~\ref{final_lem_2}. Indeed, by Proposition~\ref{psi+delta} we can take $n_1\leq \Delta$ and $\tilde{x}\in B_N$ such that
		\begin{equation*}
		\forall i\in [n_1, f(\sigma_2(\tilde{k}, g(n_1)))]\,\left(  v_n\leq \dfrac{1}{16(k+1)^2(f(n)+1)+1}  \,\land \,r_n \leq \frac{1}{16(k+1)^2}\right),
		\end{equation*}
		where the sequences are defined as before. With $n_2:=g(n_1)$, Claim 2 still holds true, just as Claim 3. Therefore, applying Lemma~\ref{final_lem_2}, with $n_3:=\sigma_2(\tilde{k}, n_2)\leq \Phi_2(k,f)$ we conclude
		\begin{equation*}
		\forall i\in[n_3, f(n_3)]\, \left( \|x_i-\tilde{x}\|\leq \frac{1}{2(k+1)}\right),
		\end{equation*}
		and the result follows again by triangle inequality.
	\end{proof}

	\begin{finalremarks}\hfill\\
		\indent In the previous two theorems, one can drop the apparent restriction to monotone functions $f:\N\to\N$ and conclude the result with the metastability bound $\Phi_i(k,f^{\rm maj})$, for $i\in\{1,2\}$, i.e. by working with the monotone function $f^{\rm maj}$ in place of the possibly non-monotone function $f$.\\
		\indent From the metastable property for the iteration $(x_n)$, we conclude that $(x_n)$ is a Cauchy sequence. Hence it converges and by Lemma~\ref{asymp_reg}, it must converge to a zero of the maximal monotone operator. This proof avoids the use of sequential weak compactness and only requires the weaker form \eqref{proj} of the projection argument. Furthermore, the extracted information is highly uniform in the parameters of the problem, only depending on the exhibited functions {\rm a}, {\rm A} (or {\rm A'}), {\rm b}, {\rm B}, {\rm E} and on the natural numbers ${\cal E}$ and ${\cal D}$. We can allow for $(\alpha_n)\subset [0,1]$: no real restriction on having $\alpha_n >0$ is needed, which is highlighted by the fact that no function witnessing the positivity of $\alpha_n$ was required for the final bounds; since $\alpha_n \to 0$, the equivalence {\rm (C\ref{c2})} -- \eqref{c2'} holds even without the condition $\alpha_n<1$.
	\end{finalremarks}

\section*{Acknowledgments}
I would like to thank Lauren\c{t}iu Leu\c{s}tean, who hosted me in Bucharest in February 2018, where most of the work in this paper was carried out, and for suggesting me to look at the Halpern type proximal point algorithm. The paper also benefited from discussions with Fernando Ferreira, Ulrich Kohlenbach and Bruno Dinis.

The author acknowledges and is thankful for the financial support of: FCT - Funda\c{c}\~ao para a Ci\^{e}ncia e Tecnologia under the project UID/MAT/04561/2019; the research center Centro de Matem\'{a}tica, Aplica\c{c}\~{o}es Fundamentais com Investiga\c{c}\~{a}o Operacional, Universidade de Lisboa; and the `Future Talents' short-term scholarship at Technische Universit{\"a}t Darmstadt.

\bibliographystyle{siam}
\bibliography{References}

\end{document}